\documentclass{amsart}

\usepackage{amsmath,amssymb,amsthm, amsfonts}
\usepackage{array,mathtools} 
\usepackage{mathrsfs}
\usepackage{euler}
\usepackage{times}
\usepackage[T1]{fontenc}

\usepackage{tikz-cd}

\usepackage{hyperref}

\newcommand{\bA}{\mathbb{A}}

\newcommand{\bL}{\mathbb{L}}

\newcommand{\bV}{\mathbb{V}}

\def\ddefloop#1{\ifx\ddefloop#1\else\ddef{#1}\expandafter\ddefloop\fi}
\def\ddef#1{\expandafter\def\csname c#1\endcsname{\ensuremath{\mathcal{#1}}}}
\ddefloop ABCDEFGHIJKLMNOPQRSTUVWXYZ\ddefloop

\newcommand{\rC}{\mathrm{C}}

\newcommand{\red}{\mathrm{red}}

\newcommand{\qcoh}{\mathrm{D}_\mathrm{qcoh}}
\newcommand{\coh}{\mathrm{D}_\mathrm{coh}}
\DeclareMathOperator{\sym}{Sym}
\DeclareMathOperator{\ext}{Ext}
\DeclareMathOperator{\Char}{char}

\newcommand{\at}{\mathrm{At}}

\newcommand{\spec}{\mathrm{Spec}}

\newcommand{\id}{\mathrm{id}}

\newcommand{\Hom}{\mathop{\mathcal{H}om}\nolimits}
\newcommand{\tr}{\mathop{\rm tr}\nolimits}

\newcommand{\grpd}{\mathbf{Grpd}}
\newcommand{\sets}{\mathbf{Sets}}
\newcommand{\sch}{\mathbf{Sch}_k}
\newcommand{\dm}{\mathbf{DM}_k}

\newcommand{\pd}[1]{h^1/h^0(#1^\vee)}

\theoremstyle{plain}

\newtheorem{thm}{Theorem}[section]
\newtheorem{prop}[thm]{Proposition}
\newtheorem{lem}[thm]{Lemma}
\newtheorem{lemma}[thm]{Lemma}
\newtheorem{cor}[thm]{Corollary}

\theoremstyle{remark}
\newtheorem{rmk}[thm]{Remark}

\title{Simple obstructions and cone reduction}

\begin{document}

\begin{abstract}
Let $X$ be a Deligne-Mumford stack locally of finite type over an algebraically closed 
field $k$ of characteristic zero. We show that the
intrinsic normal cone $C_X$ of $X$ is supported in the subcone  $\bV(\Omega_X[-1])$ 
($h^1/h^0((\Omega^1_X)^\vee)$) of its intrinsic normal sheaf $N_X$. This leads to an alternative proof 
of cone reduction by cosections for $C_X$.
 We also discuss vanishing of simple obstructions
under the Buchweitz-Flenner semiregularity map for sheaves.
\end{abstract}


\author{F. Qu}
\address{Northeast Normal University, School of Mathematics and Statistics, Changchun, Jilin, China}
\email{quf996@nenu.edu.cn}

\maketitle

\section{Introduction}
For a closed immersion between schemes  $f\colon X \to Y$, we have 
the normal sheaf $N_f$ and normal cone $C_f$ defined by
\[
N_f=\spec_{\cO_Y}\sym I/I^2,  C_f=\spec_{\cO_Y}\oplus_{d\ge 0} I^d/I^{d+1},
\] where $I$ denotes the ideal sheaf of $X$ in $Y$.
As functors of closed immersions between schemes locally of finite type over a base field $k$, the functors
$N$ and $C$ are extended to morphisms between algebraic stacks (\cite{Kre1, Kre2, KKP,AP}), and are
involved in the construction of virtual fundamental classes \cite{LT, BF}. See \cite{BH} and the references therein for extensions beyond higher stacks.

For a DM stack $X$ locally of finite type over $k$, the normal sheaf and normal cone of the
map $X \to \spec k$ are the intrinsic normal sheaf and intrinsic normal cone $N_X$ and $C_X$ introduced in
\cite{BF}. And it is proved that they carry information about deformations of maps from $\spec k$ to $X$.
More precisely, for any $k$-point 
\[
\xi \colon \spec k \to X,
\] the coarse moduli space of the pullback of $N_X$
along $\xi$ is $\ext^1_k(\xi^*\bL_X, k)$,  where $\bL_X$ denotes the cotangent complex of $X$.
The vector space 
$\ext^1_k(\xi^*\bL_X, k)$ is the universal obstruction space for $\xi$ (\cite[Lemma 4.6]{BF}),
and its subspace of  
curvilinear obstructions is 
the coarse moduli space of the pullback of $C_X$
along $\xi$ (\cite[Lemma 4.7]{BF}). Globally, universality of $N_X$ builds into the definition of an obstruction theory for $X$ and it is equivalent to a closed embedding of
$N_X$ into a cone stack (\cite[Theorem 4.5]{BF}), while curvilinearity of $C_X$ 
lacks transparent description.

When $\Char k=0$, curvilinear extensions are {\em simple extensions}.(See Section \ref{simex}.)
The key observation in this paper is that, {\em simple obstructions}, i.e., obstructions for simple extensions,
are contained in the subspace 
\[
\ext^1_k(\mathbf{L} \xi^*\Omega^1_X, k)
\]
of $\ext^1(\xi^*\bL_X, k)$, and we are led to consider the subcone 
\[
\bV(\Omega_X^1[-1]),\text{or} \  h^1/h^0((\Omega^1_X)^\vee), 
\]
of $N_X$.(See Section \ref{V} for notation.) 
If $k$ is also algebraically closed, then we show $C_X$ is supported in $\bV(\Omega_X^1[-1])$
 (Theorem \ref{abs}). It is proved using
Lemma \ref{loc} which deduce global information about support of closed substacks from their pointwise coarse sheaves. Here the support of a closed substack $C$ is the reduced substack $C_\red$ associated to it.
 
 Cone reduction for a subcone of $N_X$ concerns its support. 
It is easy to see cone reduction by cosections for $\bV(\Omega_X^1[-1])$.
A cosection 
\[
\sigma\colon N_X \to A^1_k
\]
is a map between cones and is induced by a map 
$\cO_X[1] \to \bL_X^{\ge -1}$
from $\cO_X[1]$ to the truncated cotangent complex of $X$ 
 in $\qcoh(X)$.
Since there is no nonzero map from $\cO_X[1]$ to $\Omega^1_X$[0], we conclude
$\bV(\Omega_X^1[-1])$ is supported in the kernel $\sigma^{-1}\{0\}$. This observation improves  
\cite[Corollary 4.5]{KL}. Note that we consider maps $\cO_X[1] \to \bL_X^{\ge -1}$ without
additional obstruction theories for $X$.
More generally, for any $\cG[1] \to \bL_X^{\ge -1}$, where $\cG$ is a coherent sheaf on $X$,  the 
obvious vanishing 
of $\cG[1] \to \Omega_X^1[0]$ can be viewed as a global version of \cite[Proposition 6.13(1)]{BuF}.
Semiregularity can also be treated globally from this perspective as in \cite[Section 5]{Cha} and \cite[Section 3]{Sch}.

We also deduce a relative version.
For a map  $f\colon X \to Y$ between DM stacks locally of finite type over $k$, and  $Y$ smooth, 
we look for subcones of $N_f$ that contain $C_f$.
The cone $\bV(\Omega_f[-1])$ is too small for closed immersions.
As $C_X$ is given by $[C_f/f^*T_Y]$, it follows from Theorem \ref{abs} that 
$C_f$ is supported in \[
\bV(\{f^*{\Omega_Y^1}\to \Omega_X^1\}[-1])\]
(Proposition \ref{rel}) . 
Here
$
\{f^*\Omega_Y^1 \to \Omega_X^1\}
$ is the two-term complex with $f^*\Omega_Y^1$ in degree $-1$ and $\Omega_X^1$ degree 0.

Cone reduction by cosections can also be treated in the framework of derived algebraic geometry and
derived differential geometry. See
\cite[Section 6]{AKLPR}, \cite[Appendix A.1]{BKP}, and \cite[Proposition 5.1]{Sa}.

Recent interests in cone reduction come from the work of 
Oh-Thomas (\cite{OT}).
In their approach to construct algebraic virtual classes on moduli spaces of sheaves on complex Calabi-Yau 4 folds, it is of interest to show
 a normal cone inside a quadratic bundle is isotropic or supported in some maximal isotropic subbundle.
 The isotropic condition in \cite{OT} is established using
the Darboux theorem (\cite{BBJ}) for shifted symplectic structures (\cite{PTVV}).
 The local counterparts to these global questions
 concern the subspace of curvilinear obstructions 
inside the standard
 obstruction space with its quadratic form induced by Serre Duality, and imply global results for the reduced normal cone. However, these local questions are not as trivial to establish as those in this paper and seems out of reach for us.

The paper is organized as follows.
Preliminary results are collected in Section 2, from which the
the main results in Section 3 follow rather easily. In Section 4, we globalize
the semiregularity map(\cite{BuF}) for sheaves 
to a map in the derived category, and show vanishing of obstructions for simple extensions.

\subsection*{Notation and convention}
We work over a base field $k$.
Denote $\sets$ the category of sets, and $\grpd$ the category of groupoids.
Denote $\dm$ the category of Deligne-Mumford stacks locally of finite type over $k$,
and $\sch$ its full subcategory of schemes in $\dm$.
For $X\in \dm$ and  a ring $R$ of finite type over $k$, $X(R)$ denotes the set of morphisms $\spec R \to X$ in $\dm$.
Functors for sheaves and complexes are derived. For sheaves, we will emphasis left or right derived functor 
by adding $\mathbf{L}$ or $\mathbf{R}$.

\section{Preliminaries}

\subsection{Coarse sheaf and support}

\subsubsection{$\pi_0$}
Let $\pi_0\colon \grpd \to \sets$ be the functor that sends a groupoid to its isomorphic classes of objects.
It is the
left adjoint of the inclusion $\sets \to \grpd$.

We view $\sch$ as a site with the big \'etale topology.
For a stack 
\[
Y\colon \sch \to \grpd,
\] denote $\pi_0(Y)
$ the sheafification of the 
composition 
\[
\pi_0\circ Y\colon \sch \to \sets.
\]
Clearly, $\pi_0(Y)$ is functorial in $Y$ and there is a natural transformation
\[
Y \to \pi_0(Y).
\]

Recall a map $ M \to N$ between stacks  over $\sch$ is fully faithful if the relative diagonal 
$M \to M\times_N M$ is an isomorphism. Equivalently, for any $T \in \sch$,
$M(T) \to N(T)$ is fully faithful in $\grpd$.
\begin{lem}\label{car}
Let $M \to N$ be a fully faithful map of stacks, then
\[
\begin{tikzcd}
M \ar[r]\ar[d] & N \ar[d]\\
\pi_0(M) \ar[r] & \pi_0(N)
\end{tikzcd}
\] is a cartesian diagram in the $(2,1)$-category of stacks over $\sch$.
\end{lem}
\begin{proof}
If
\[
G \to H
\] is a fully faithful map in $\grpd$, 
then
\[
\begin{tikzcd}
G \ar[r]\ar[d] & H \ar[d]\\
\pi_0(G) \ar[r] & \pi_0(H)
\end{tikzcd}
\] is (homotopy) cartesian in $\grpd$.
It follows that 
\[
\begin{tikzcd}
M \ar[r]\ar[d] & N \ar[d]\\
\pi_0\circ M \ar[r] & \pi_0\circ N
\end{tikzcd}
\] is cartesian, and we see the lemma holds by taking stackification.
\end{proof}

\begin{rmk}
If $g\colon Z \to Y$ is a closed immersions between algebraic stacks, then $g$ is fully faithful.
From the lemma we see that $\pi_0(Z) \to \pi_0(Y)$ is representable by closed immersions.
In particular, $\pi_0(Z) \to \pi_0(Y)$ is fully faithful, and we view $\pi_0(Z)$ as a subsheaf of $\pi_0(Y)$.
\end{rmk}

\begin{rmk}
For the closed immersion $C_X \to N_X$,  the lemma is proved during the proof of \cite[Proposition 2.2]{Be}.
\end{rmk}

\subsubsection{Support of closed substacks}

Let $N$ be an algebraic  stack over $X \in \dm$, and $C,D$ closed substacks of $N$.
For any point $f\colon T \to X $ in $\dm$, we have subsheaves
$\pi_0(f^*C)$ and $\pi_0(f^*D)$ of $\pi_0(f^*N)$. Here $f^*C, f^*D$ are the fiber products $T\times_X C, T\times_XD$.

The following lemma is a variation on \cite[Proposition 2.1]{CL}.
\begin{lemma}\label{loc}
Let $k$ be an algebraic closed field. If for any $k$- point 
\[
\xi\colon \spec k \to X,
\] the sheaf $\pi_0(\xi^*C)$ is contained in $\pi_0(\xi^*D)$ as subsheaves of $\pi_0(\xi^*N)$,
then $C$ is supported in $D$. In other words, the inclusion map $C_\red \to N$ factors through the inclusion $D \to N$.
Here $C_\red$ denotes the reduced stack associated to $C$.
\end{lemma}
\begin{proof}
By passing to smooth covers, we can assume $N$ and $X$ are in $\sch$.
For schemes in $\sch$, closed points are dense, and $C$ is supported in $D$
if
closed points of $C$ belongs to those of $D$.
As $k$ is algebraically closed, closed points are $k$-points, and we only need to show
$C(k) \subset D(k)$.

Any $k$-point of $C$ lies over some closed point $\xi\colon \spec k \to X$.
By Lemma \ref{car}, we have cartesian diagrams
\[
\begin{tikzcd}
\xi^*C\ar[r]\ar[d]  &  \xi^*N\ar[d]\\
\pi_0(\xi^*C) \ar[r] & \pi_0(\xi^*N),
\end{tikzcd}
\begin{tikzcd}
\xi^*D\ar[r]\ar[d]  &  \xi^*N\ar[d]\\
\pi_0(\xi^*D) \ar[r] & \pi_0(\xi^*N).
\end{tikzcd}
\]
From the assumption $\pi_0(\xi^*C) \subset \pi_0(\xi^*D)$,
we see $\xi^*C \subset \xi^*D$ as closed subschemes of $\xi^*N$, and
closed points of $C$ over $\xi$ belong to $D(k)$.

\end{proof}
 \subsection{The functor $\pd{(-)}$} \label{V}
 Let $X \in \dm$. Associated to $E\in \coh^{\le 0}(X)$ is the cone stack 
 \[
 \pd{E}.
 \](See \cite[Proposition 2.4]{BF}.)
If $\tau^{[-1,0]}(E)$ is quasi-isomorphic to a two-term complex of coherent sheaves 
\[
\cQ \to \cF
\] with $\cF$ locally free, then $h^1/h^0(E^\vee)$ is the quotient stack
\[
[\rC(\cQ)/\rC(\cF)],
\] where $\rC(-)=\spec_{\cO_X}\sym(-)$.
The map $\cQ\to \cF$ induces a map between abelian cones $\rC(\cF) \to \rC(\cQ)$,
and as a morphism of abelian group schemes over $X$, we have a Picard stack $[\rC(\cQ)/\rC(\cF)]$.
It can also be described as the stack which associates to
a $T$-point $f\colon T \to X$ the mapping space between $f^*(E[-1])$ and $\cO_T$, and 
is denoted $\bV(E[-1])$ in \cite{AP}. 
We will also use the $\bV$ notation.

\begin{rmk}
The $\bV$ construction using mapping spaces is more general than $h^1/h^0$. For a 3-term complex
$F$
in degree $[0,2]$ on an algebraic stack, $\bV(F)$ is in general
a Picard 2-stack. The $\bV$ notation depends on the convention of
authors, it is either $\bV(E)$ or $\bV(E^\vee)$.
\end{rmk}

When $E$ is perfect, $\pi_0(\pd{E})$ is give by $h^1(E^\vee)$. Here $h^1(E^\vee)$
denotes its extension to the big \'etale site. (See e.g., \cite[Section 2, p.~60]{BF} for extensions from
 a small site to a big site.)

From its description, it is clear the pullback of $\pd{E}$
 along $g\colon Y \to X$ is $h^1/h^0((g^*E)^\vee)$.
In particular, for any point $\xi\in X(k)$, the pullback is $\pd{(\xi^*E)}$.
As $\tau^{[-1,0]}(\xi^*E)$ is quasi-isomorphic to $h^{-1}(\xi^*E)[1]\oplus h^0(\xi^*E)$
in the derived category of $k$-modules, it is easy to see
$\pi_0(h^1/h^0((g^*E)^\vee))$ is isomorphic to $\rC(h^{-1}(\xi^*E))$, or the $k$-vector space
$\ext^1_k(\xi^*E, k)$.

\subsection{Simple extensions and obstructions} \label{simex}
A closed immersion $S \to S'$ between schemes with ideal $I$ is 
a square zero extension if $I^2=0$, in which case the ideal sheaf $I$
has an induced $\cO_S$-module structure.
In the affine case, $S \to S'$ corresponds to a surjection of commutative rings $A' \to A$ with
square zero kernel.

Isomorphic classes of square zero extensions of $S$ by an $\cO_S$-module $J$
can be identified with 
$
\ext^1_S(\bL_S, J)
$(\cite[III.1.2.3]{IL}).
For a square zero extension $S \to S'$, the element it determines in
$\ext^1_S(\bL_S, J)$
is the map
\[
\bL_S \to \bL_{S/S'} \to \tau^{\ge-1}\bL_{S/S'}\simeq J[1].
\]

The natural map $\bL_S \to \Omega_S^1$ induces an inclusion
\[
\ext^1_S(\Omega_S^1, J) \to \ext^1_S(\bL_S,J).
\]
Elements of the subspace
$
\ext^1_S(\Omega_S^1, J)
$ correspond to simple extensions.
Alternatively, a square zero extension $S \to S'$ is simple, if
\[
I/I^2 \to {\Omega_{S'}^1}_{|_S} \to \Omega_S\to 0
\] is exact on the left. (See e.g., \cite[Theorem 9.2 (2)]{BuF}.)

It is easy to see curvilinear extensions
\[
k \to k[t]/(t^{n+1}) \to k[t]/(t^n), n\ge 2
\] are simple when $\Char k=0$.

Let $S \to S'$ be a square zero extension in $\sch$
classified by some map $e\colon \bL_S \to J[1]$,
 and $f\colon S \to X$ a map  in $\dm$.
The obstruction to lift $f$ from $S$ to $S'$ 
lies in $\ext^1_S(f^*\bL_X, J)$, and
is given by
the map 
\begin{equation} \label{obsc}
f^*\bL_X \to \bL_S \xrightarrow{e} J[1].
\end{equation}(\cite[III.2.2.4]{IL})

\begin{lemma} 
Let $e$ be a simple extension of $S$ by $J$,  then the obstruction 
to lift $f\colon S \to X$ along $e$ to $S'$
lies in the subspace
\[
\ext^1_S(\mathbf{L}f^*\Omega^1_X, J)
\] of $\ext^1_S(f^*\bL_X, J)$.
\end{lemma}

\begin{proof}
The lemma holds because the map $f^*\bL_X \to \bL_S \to \Omega_S^1[0]$ factors through
$\mathbf{L}f^*\Omega^1_X$.
\end{proof}
\begin{rmk}\label{sim}
Let $\xi\colon \spec k \to X$ be a $k$-point of $X$.
Let $A$ be an Artinian $k$-algebra with residue field $k$.
and $f\colon \spec A \to X$ a deformation of $\xi$.
Apply the lemma to $S=\spec A$ and $J$ the residue field of $A$, 
we see that the obstruction class belongs to
\[
\ext^1_A(\mathbf{L}f^*\Omega^1_X, k)\simeq \ext^1_k(\mathbf{L}\xi^*\Omega^1_X,k).
\] In particular, curvilinear obstructions associated to $\xi$ lie in $\ext^1_k(\mathbf{L}\xi^*\Omega^1_X,k)$.
\end{rmk}

\section{Cone reduction for $C_X$}
In this section, we assume the field $k$ has $\Char k=0$ and is algebraically closed.

For a map $f\colon X \to Y$ in $\dm$ with $Y$ smooth, 
we have the map between
sheaves of K\"ahler differentials $f^*{\Omega_Y^1} \to \Omega_X^1$, and its cone
\[
\{f^*{\Omega_Y^1} \to \Omega_X^1\}
\] in $\qcoh(X)$.

\begin{lem}
 Let $f\colon X \to Y$ be a map in $\dm$ with $Y$ smooth. Then
\[
\bV(\{f^*{\Omega_Y^1} \to \Omega_X^1\}[-1])
\] is a closed subcone of
$N_f$. 
In particular, for any $X \in \dm$, $\bV(\Omega_X^1[-1])$
is a closed subcone of $N_X$.
\end{lem}

\begin{proof}
As $Y$ is smooth, $\bL_Y\simeq \Omega_Y^1$, and 
$\mathbf{L}f^*{\Omega_Y^1} \simeq f^*\Omega_Y^1$.

The map 
$\bL_X  \to \Omega_X^1[0]=\tau^{\ge 0}\bL_X $
induces a map between distinguished triangles
\[
\begin{tikzcd}
\mathbf{L}f^*{\Omega_Y^1}  \ar[r]\ar[d,"\simeq"]
	& \bL_X \ar[r]\ar[d]
			&\bL_f \ar[d,dashed]\\
f^*{\Omega_Y^1}  \ar[r]
	&\Omega_X^1 \ar[r]
		& \{ f^*{\Omega_Y^1} \to \Omega_X^1 \},
\end{tikzcd}
\]
and we see that the cone of 
\[
\bL_f \to\{ f^*{\Omega_Y^1} \to \Omega_X^1\}
\] is isomorphic to $\tau^{\le -1}\bL_X[1]$ which lies in $\qcoh^{\le -2}(X)$.
 Now we apply 
\cite[Proposition 2.6]{BF} to finish the proof. 
 
 \end{proof}

\begin{thm}\label{abs}
Let $X$ be a DM stack locally of finite type over $k$, 
then $C_X$ is supported in $\bV(\Omega_X^1[-1])$.
\end{thm}

\begin{proof}

Consider the subcones $C_X$ and $\bV(\Omega_X^1[-1])$ of $N_X$.
By \cite[Lemma 4.7]{BF} and Remark \ref{sim},
we can apply Lemma \ref{loc} to conclude the proof.
\end{proof}

\begin{rmk}
Let  $\cX$ be an algebraic stack locally of finite type over $k$, and
$U\to \cX$ a smooth cover of $\cX$ by $U\in \sch$.
Then the normal cone
$C_\cX$ and normal sheaf $N_\cX$(\cite{AP}) are quotients of $C_U$ and $N_U$ by
the group stack $BT_{U/\cX}$ respectively. In particular,
we have a cartesian diagram
\[
\begin{tikzcd}
C_U \ar[r]\ar[d] & N_U\ar[d]\\
C_\cX\ar[r]        & N_\cX,
\end{tikzcd}
\]
and it follows from the theorem for $U$ that $C_\cX$ is 
supported in $\bV(\tau^{\ge 0}\bL_\cX[-1])$.
\end{rmk}

\begin{prop}\label{rel}
Let $f\colon X \to Y$ be a map in $\dm$ with $Y$ smooth.
Then $C_f$ is supported in $\bV(\{f^*{\Omega_Y^1} \to \Omega_X^1\}[-1])$.
\end{prop}

\begin{proof}
As $Y$ is smooth, we have an exact sequence of cones
\[
f^*T_Y \to C_f \to C_X
\] by \cite[Proposition 3.1]{BF}.

The distinguished triangle
\[
\Omega_X^1 \to  \{ f^*{\Omega_Y^1} \to \Omega_X^1\} \to f^*\Omega_Y^1[1]
\]
induces an exact sequence of cones
\[
f^*T_Y \to 
\bV(\{f^*{\Omega_Y^1} \to \Omega_X^1\}[-1])
\to \bV(\Omega_X^1[-1])
\] by\cite[Proposition 2.7]{BF}.

Note that both sequences embed into the exact sequence
\[
f^*T_Y \to N_f \to N_X
\] induced by $\bL_X \to \bL_f \to \mathbf{L}f^*\Omega_Y[1]$.
 As $C_X$ is supported in $\bV(\Omega_X^1[-1])$,
we see $C_f$ is supported in
$\bV(\{f^*{\Omega_Y^1} \to \Omega_X^1\}[-1])$

\end{proof}

\begin{rmk}
A more direct proof of 
Theorem \ref{abs}  and Proposition \ref{rel}
 is to show for a closed immersion $f\colon X \to Y$ in $\sch$ with $Y$ smooth, the normal cone $C_f$ is
supported in the abelian cone $\rC(K)$, where 
$K$ is
the kernel of the map $f^*\Omega_Y^1 \to \Omega_X^1$.
Unfortunately, we don't know how to prove this inside $\sch$.
\end{rmk}

\begin{cor}
Let $f\colon X \to Y$ be a map in $\dm$ with $Y$ smooth.
Let $\cO_X[1] \to \bL^{\ge -1}_f$ be a cosection inducing 
\[
\sigma\colon N_f \to \bA_k^1,
\]
If $\cO_X[1] \to \bL_f^{\ge -1}$ lifts to a cosection $\cO_X[1] \to \bL_X^{\ge -1}$, or equivalently
the cosection $\sigma$ descends to $N_X$,
then $C_f$ is supported in $\sigma^{-1}(0)$.
\end{cor}
\begin{proof}
Note that the map $\cO_X[1] \to \bL_f^{\ge -1} \to \{ f^*{\Omega_Y^1} \to \Omega_X^1\}$ is zero, 
as it factors through $\cO_X[1] \to \bL_X^{\ge -1} \to \Omega_X^1$ in the diagram
\[
\begin{tikzcd}
&\cO_X[1]\ar[dl]\ar[d]\\
\bL_X^{\ge -1} \ar[r] \ar[d]
	& \bL_f^{\ge -1} \ar[d]\\
\Omega^1_X \ar[r]
	& \{ f^*{\Omega_Y^1} \to \Omega_X^1\}.
\end{tikzcd}
\]
\end{proof}


\begin{rmk}
The strategy to relate $C_f$ to $C_X$  in Proposition \ref{rel}
and the condition that $\sigma$ descends in the corollary appeared in
\cite[Section 3.5]{MPT} for the case of stable pairs moduli spaces.
\end{rmk}

\section{semiregularity map for sheaves}

\subsection{}
Let $M$ be a DM stack in $\dm$ with an obstruction theory
$F \to \bL_M$, i.e., a map in $\coh(M)$ whose cone lies in $\coh^{\le -2}(M)$.
Let $\rho\colon G \to F$ be a map in $\coh(M)$,
it functions as a global semiregularity map.
For any map $f\colon S \to M$ and square zero extension $\bL_S \to J[1]$
we obtain the map
\[
\ext^1_S(f^*F, J) \xrightarrow{f^*\rho\circ} \ext^1_S(f^*G, J),
\] which we view as a semiregularity type map.

\subsection{}
Now we move on to the semiregularity map for sheaves in \cite{BuF}, we will globalize the map as above. 
For a discussion of works related to semiregular maps, see \cite[Chapter 3]{Le}.

Let $\pi\colon X \to B$ be a flat family of smooth projective varieties over an algebraically closed  
field $k$ with $\Char k=0$, and 
$\mu\colon M \to B$ a moduli 
scheme of stable sheaves on the fibers of $\pi$. (See e.g., \cite[Chapter 4]{HL}.)
Consider the cartesian digram,
\[
\begin{tikzcd}
M\times_B X \ar[r,"q"]\ar[d,"p"] & X \ar[d,"\pi"] \\
M \ar[r,"\mu"]    & B,
\end{tikzcd}
\] where we used $p,q$ to denote projection maps.

We have a map 
\begin{equation} \label{obs}
(p_*\Hom(E,E)[1])^\vee  \to \bL_{\mu}
\end{equation}
obtained from 
the Atiyah class of a twisted universal sheaf $E$ 
on $M\times_BX$ relative to $q$.
(See e.g., \cite[Section 1.4]{Ku}.) And an obstruction theory can be constructed from this map (\cite[Section 4]{HT}).

Next, we will construct the map in $\coh(M)$ with target $(p_*\Hom(E,E)[1])^\vee$,
 or dually, a map with source $p_*\Hom(E,E)[1]$, that induces the local semiregularity map
up to $\pm$ signs.
For each integer $n\ge 0$, 
consider the composition 
\[
 E \otimes \wedge^n \bL_p[n]
\xrightarrow{\at(E)\otimes \id} E\otimes \bL_p[1]\otimes \wedge^n \bL_p[n]
\xrightarrow{\id \otimes \wedge} E\otimes \wedge^{n+1} \bL_p[n+1],
\]
where 
$
\at(E)\colon E \to E\otimes \bL_p[1]
$ denotes the Atiyah class of $E$ relative to $p$.
It induces a map
\[
\at(E)\circ \colon \Hom(E, E \otimes \wedge^n \bL_p[n]) \to \Hom(E, E \otimes \wedge^{n+1} \bL_p[n+1]).
\]
These maps determine
\[
(\at(E)\circ)^n[1] \colon \Hom(E, E)[1] \to \Hom(E, E\otimes \wedge^n \bL_p[n])[1].
\]
Composing with the trace map
\[
\tr\colon \Hom(E, E \otimes \wedge^n \bL_p[n])[1] \simeq \Hom(E,E)\otimes \wedge^n \bL_p[n+1]
\xrightarrow{\tr_E \otimes \id }  \wedge^n \bL_p[n+1]
\] and pushforward alone p, 
we obtain the map
\begin{equation} \label{sr}
p_*\Hom(E,E)[1] \to p_* \wedge^n \bL_p[n+1].
\end{equation}
The product of these maps over $n$
\[
p_*\Hom(E,E)[1] \to \prod_{n\ge 1} p_* \wedge^n \bL_p[n+1]
\] in $\coh(M)$
 is dual to the desired global semiregularity map up to $\pm$ signs on each factor indexed by $n$.
 
 As 
$
\bL_p\simeq  q^*\bL_\pi\simeq q^*\Omega_\pi^1
$,
we see
\[
p_* \wedge^n \bL_p[n+1] \simeq p_* q^*\Omega^n_\pi[n+1]
\simeq 
\mu^*\mathbf{R}\pi_*\Omega_\pi^n[n+1].
\]
Note that 
\[
\mathbf{R}\pi_*\Omega_\pi^n
\simeq
\bigoplus_{j \ge 0} \mathbf{R}^j\pi_*\Omega_\pi^n[-j]
\] is a direct sum of shifted 
locally free coherent sheaves by \cite[Theorem 6.1, Proposition 5.5(i)]{De}.

Vanishing of simple extensions under the semiregularity map was proved
in \cite{BuF}, an alternative proof for smooth projective families
follows from the lemma below.

\begin{lemma}
Let $M$ be a DM stack, $m$ an integer, $\cG$ a locally free coherent sheaf on $M$ with a 
map $\phi \colon \cG[m] \to \bL _M$.
Let 
$f\colon \spec A \to M$ be a map.
If $m \ne 0$,
then the composition 
\[
f^*\cG[m] \xrightarrow{f^*\phi} f^*\bL_M \to \bL_A \to \Omega_A^1
\] is zero.
It follows that for any $A$-module $J$ and map $e\colon \Omega_A^1 \to J[1]$ in $\qcoh(\spec A)$
the composition 
\[
f^*\cG[m] \xrightarrow{f^*\phi} f^*\bL_M \to \bL_A \to \Omega_A^1 \xrightarrow{e} J[1]
\] is zero.
\end{lemma}
\begin{proof}
As $f^*\cG$ is a locally free coherent sheaf on $\spec A$, 
\[
\hom_{\qcoh(\spec A)}(f^*\cG[m], \Omega_A^1)=\ext^{-m}_A(f^*\cG, \Omega_A^1)=0
\] if $m\ne 0$.  For $m=0$,
\[
\ext^{-m}_A(f^*\cG, J[1])=\ext^{1}_A(f^*\cG, J)=0.
\]
\end{proof}

\end{document}